\documentclass[11pt,reqno]{amsart}
\usepackage{amsmath,amssymb,amsthm}
\usepackage[dvipsnames]{xcolor}
\usepackage[colorlinks=true, linkcolor=BrickRed, urlcolor=Blue,citecolor=OliveGreen,menucolor=]{hyperref}
\usepackage{bbm}
\usepackage{enumitem}

\DeclareMathOperator\Lie{\textnormal{Lie}}

\DeclareMathOperator{\ad}{\textnormal{ad}}
\DeclareMathOperator{\Ad}{\textnormal{Ad}}
\DeclareMathOperator{\Ran}{\textnormal{Ran}}
\DeclareMathOperator{\spn}{\textnormal{span}}
\DeclareMathOperator{\Aut}{\textnormal{Aut}}
\DeclareMathOperator{\id}{\textnormal{id}}

\newcommand{\R}{\mathbb{R}}
\newcommand{\C}{\mathbb{C}}

\renewcommand{\H}{\mathcal{H}}
\newcommand{\Id}{\mathbbm 1}

\newcommand{\M}{\mathcal{M}}  

\newcommand{\N}{\mathcal{N}}  

\newcommand{\I}{N}  

\newcommand{\lgk}{\mathcal{A}(G,K)}  
\newcommand{\lcgk}{\mathcal{AC}(G,K)}  

\newcommand{\J}{\mathcal{J}}

\newcommand{\V}{J}  

\begin{document}

\newtheorem{thmx}{Theorem}
\renewcommand{\thethmx}{\Alph{thmx}} 

\newtheorem*{theorem*}{Theorem}
\newtheorem{theorem}{Theorem}
\newtheorem{corollary}[theorem]{Corollary}
\newtheorem{criterion}[theorem]{Criterion}
\newtheorem{lemma}[theorem]{Lemma}
\newtheorem{proposition}[theorem]{Proposition}

\theoremstyle{definition}
\newtheorem{definition}[theorem]{Definition}
\newtheorem{example}[theorem]{Example}
\newtheorem{notation}[theorem]{Notation}
\newtheorem{problem}[theorem]{Problem}
\newtheorem{remark}[theorem]{Remark}
\numberwithin{theorem}{section}
\numberwithin{equation}{section}

\title{Nijenhuis operators on Banach homogeneous spaces}

\author{Tomasz Goli\'nski}
\address{University of Bia{\l}ystok, Faculty of Mathematics\\
    Ciołkowskiego 1M, 15-245 Bia{\l}ystok, Poland}
\email{tomaszg@math.uwb.edu.pl}

\author{Gabriel Larotonda}
\address{Universidad de Buenos Aires and CONICET, Argentina}
\email{glaroton@dm.uba.ar}

\author{Alice Barbora Tumpach}
\address{Institut CNRS Pauli, Wien, Technische Universität Wien, and Laboratoire Paul Painlev\'e, Villeneuve d'Ascq}
\email{barbara.tumpach@math.cnrs.fr}



\begin{abstract}For a Banach--Lie group $G$ and an embedded Lie subgroup $K$ we consider the homogeneous Banach manifold $\M=G/K$. In this context we establish the most general conditions for a bounded operator $\I$ acting on $Lie(G)$ to define a homogeneous vector bundle map $\N:T\M\to T\M$. In particular our considerations extend all previous settings on the matter and are well-suited for the case where $Lie(K)$ is not complemented in 
 $Lie(G)$. We show that the vanishing of the Nijenhuis torsion for a homogeneous vector bundle map $\N:T\M\to T\M$ (defined by an admissible bounded operator $\I$ on $Lie(G)$) is equivalent to the Nijenhuis torsion of $\I$ having values in $Lie(K)$. As an application, we consider the question of integrability of an almost complex structure $\J$ on $\M$ induced by an admissible bounded operator $\V$, and we give a simple characterization of integrability in terms of certain subspaces of the complexification of $Lie(G)$ (which are not eigenspaces of the complex extension of $\V$).
\end{abstract}

\keywords{Nijenhuis operator, homogeneous space, almost complex manifold, Banach--Lie group}
\subjclass[2020]{58B12, 53C30, 32Q60, 58B20, 53C15}

\maketitle

\tableofcontents

\section{Introduction}

For a smooth vector bundle map $\N:T\M\to T\M$ (covering identity), its \emph{Nijenhuis torsion} is defined as
\[
\Omega_{\N}(X,Y)=\N [\N X,Y]+\N [X,\N Y]-[\N X,\N Y]-\N^{\; 2}[X,Y]
\]
in terms of vector fields $X,Y$ in $\M$. Here $[\cdot, \cdot ]$ is the usual Lie bracket of vector fields. Sometimes $\Omega_\N$ is called the \textit{Nijenhuis tensor} of $\N$ in the literature. It was defined in \cite{nijenhuis} in order to describe the behaviour of distributions spanned by eigenvectors of $\N$, see \cite{kosmann-bial} for a review of the history around this subject. The Nijenhuis torsion is closely related to the problem of integrability of almost complex structures solved in the finite-dimensional real-analytic case by Eckmann and Fr\"olicher in \cite{eckmann51,frolicher55} and for the smooth (or even less regular case) by Newlander and Nirenberg in \cite{newlander}.

The map $\N$ is a \emph{Nijenhuis operator} if its Nijenhuis torsion vanishes. Nijenhuis operators are useful in the study of integrable systems, see e.g. \cite{magri84,magri90,kosmann-bial,bols-bor,panasyuk2006}  and the references therein. For example the vanishing of the torsion $\Omega_{\N}$ is also equivalent to the Jacobi identity for a new \emph{deformed} bracket on vector fields of $\M$ defined as follows
\[
[X,Y]_{\N}=[\N X,Y] +[X,\N Y]-\N [X,Y]
\]
(see \cite{magri90,konyaev-n2}), and in fact $\N$ is a Lie algebra morphism from the new Lie algebra structure to the old one. It allows one also to deform a Poisson brackets on the manifold via so called Poisson--Nijenhuis structures \cite{magri90}. They are also linked with Poisson--Lie groups and even Poisson groupoids and Lie bialgebroids \cite{das19}.

There is a recent growing interest in Nijenhuis operators and their applications, as can be seen in the series of recent papers, for instance \cite{bolsinov-n1,bolsinov-na5} or \cite{grabowski2001} and references therein.

\smallskip

In a Banach homogeneous manifold $\M=G/K$, where $G$ is Banach--Lie group and $K$ an embedded Lie subgroup, a vector bundle map $\N:T\M\to  T\M$ is \emph{homogeneous} if it is invariant with respect to the natural action of the Lie group $G$ on $\M$. In this paper we are interested in homogeneous vector bundle maps that can be described by operators $\I\in \mathcal B(\mathfrak g)$ with certain properties (\textit{admissible operators} of Definition \ref{homj}), where $\mathfrak g=\Lie(G)$ is the Banach--Lie algebra of $G$. The main purpose of this paper is to prove the following 

\begin{thmx}
Let $\N:T\M\to  T\M$ be a homogeneous vector bundle map induced by an admissible operator $\I\in \mathcal B(\mathfrak g)$ and the action of $G$ in $\M$. Then $\N$ is Nijenhuis if and only if
\[
\I[v,\I w]+\I[\I v,w]-[\I v,\I w]-\I^2[v,w]\in \mathfrak k = \Lie(K)
\]
for all $v,w\in \mathfrak g$.
\end{thmx}
This is Theorem \ref{torsionJ} below; in the process, we hope to clarify certain aspects of known proofs of related results. As a corollary of Theorem \ref{torsionJ}, for homogeneous almost complex structures $\J$, we give a linear characterization in the complexification of $\mathfrak g$ for $\J$  to be integrable, 
 invoking the Banach version of the Newlander--Nirenberg theorem \cite[Theorem 7]{beltita05integrability} and our previous theorem. More precisely, let $\mathfrak g^{\mathbb C}$ be the complexification of $\mathfrak g$, and $J^{\mathbb C}$ the complexification of $J$ ($J$ is an admissible operator in $\mathfrak g$ inducing the homogeneous almost complex structure $\J$ in $G/K$, see Definition \ref{admJ}). Let 
\[
Z_+= \{v\in \mathfrak g^{\mathbb C}: J^{\mathbb C}v-iv\in \mathfrak k^{\mathbb C}\}.
\]
\begin{thmx}
The almost complex structure $\J$ is integrable if and only if $Z_+$ is a Banach-Lie subalgebra of $\mathfrak g^{\mathbb C}$.
\end{thmx}
This is done in Theorem \ref{teob} and Corollary \ref{integrability} below. Remarkably, the existing proofs of this characterization on homogeneous spaces employ certain properties of vector fields and almost complex structures that are not always satisfied; in particular the proof by Fr\"olicher in the finite dimensional setting \cite[Section 19, Satz 2]{frolicher55} involves the existence of local cross sections of the quotient map $\pi:G\to \M=G/K$ (see Remark \ref{split}), which may not exist in general in the Banach setting. On the other hand, in \cite[Theorem 13]{beltita05integrability}, which is stated in the Banach setting, there is no mention of Nijenhuis operators and the discussion concerns only almost complex structures, which involves the additional constraint $\J^2=-1$ (a particular case of our results). We also note that there is a problem with the proof of \cite[Theorem 13]{beltita05integrability} (see Remarks \ref{rem:alpha-rel}, \ref{notrel} and \ref{split} below), which our approach clarifies.

The results of this paper will also be applied in the study of almost K\"ahler structures on the coadjoint orbits of the unitary groups, \cite{GLT-kahler}.

\smallskip

Here is a short description of the organisation of this paper: in Section \ref{section:hs} we introduce the necessary ideas and objects from the theory of vector bundles and homogeneous spaces, we discuss vector fields in the homogeneous space and homogeneous vector bundle maps, we point out some possible pitfalls and illustrate them with an elementary example in the unit sphere of $\mathbb R^3$. In Section \ref{section:ntorsion}, we recall the notion of Nijenhuis torsion for a vector bundle map, and using the exponential chart of the group $G$ and what we call projected vector fields in $G/K$, we prove the first main Theorem \ref{torsionJ}. In Section \ref{section:almost}, we present the almost complex structures as special cases of the homogeneous maps discussed before, and we prove the second main result of the paper, Theorem \ref{teob}. We finish the paper discussing some examples and applications of our main theorems.

\section{Homogeneous  structures}\label{section:hs}

In this section main structures are introduced and the terminology and notations are fixed.

\subsection{Notations and basic properties}

\begin{notation}\label{not1} Let $\M_1,\M_2$ be smooth real manifolds and $E$ be a real Banach space.
\begin{itemize}
    \item If $f:\M_1\to \M_2$ is a smooth map we denote with $f_*:T\M_1\to T\M_2$ its differential, which at each point $m\in\M_1$ will be denoted $f_{*m}:T_m\M_1\to T_{f(m)}\M_2$. 
    \item We say that $f$ is an \emph{immersion} if  for all $m\in \M_1$ the map $f_{*m}$ is an injection with closed range,  and we say that $f$ is a \emph{submersion} if $f_{*m}$ is an surjection for all $m\in \M_1$.
    \item If $X$ is a vector field in $\M_1$ and $p\in \M_1$, we sometimes write $X_p$ instead of $X(p)$ for convenience.
    \item If $X_1,X_2$ are vector fields in $\M_1$, $\M_2$ respectively, they are $f$-related if $X_2(f(m))=f_{*m}(X_1(m))$ for all $m\in \M_1$. It is well-known that in this case if also a vector field $Y_2$ is $f$-related with $X_2$, then
\begin{equation}\label{lieb}
    [X_2,Y_2]\big(f(m)\big)=f_{*m}\big([X_1,Y_1](m)\big),
\end{equation}
    where $[\cdot,\cdot]$ denotes the Lie bracket of vector fields.
    \item We use $\mathcal B(E)$ to denote the space of bounded linear operators acting on $E$ and we denote with $GL(E)$ the group of invertible bounded operators.
\end{itemize}
\end{notation}

\medskip

\begin{notation} Let $G$ be a real Banach--Lie group with Banach--Lie algebra~$\mathfrak g$. 
\begin{itemize}
    \item The left and right multiplication by elements $g\in G$ will be denoted as $l_g(h):=gh$ and $r_g(h):=hg$, and the conjugation is $l_gr_{g^{-1}}$ i.e. $h\mapsto ghg^{-1}$.
    \item The differential of $l_g$ at the unit element $h=1$ will be denoted as $L_g$ i.e. $L_g=(l_g)_{*1}$ and likewise $R_g= (r_g)_{*1}$. 
\item The adjoint map on the Lie algebra (the differential of conjugation at the identity) is denoted as $\Ad_g$ i.e. $\Ad_g= L_gR_{g^{-1}}$. 
\item The Lie bracket in $\mathfrak g$ will be denoted as $[v,w]=\ad_v w$, where $\ad=(\Ad)_{*1}$ is the differential at $g=1$ of the adjoint representation of the group $\Ad:G\to  GL(\mathfrak g)$.
\end{itemize}
\end{notation}

\begin{definition}[Homogeneous spaces]\label{homs} Let $K$ be an immersed Banach--Lie subgroup of $G$ with Banach--Lie subalgebra $\mathfrak k\subset \mathfrak g$. We say that 
$G/K$ is a \emph{homogeneous space} of $G$ if the quotient space for the right action of $K$ on $G$
\[
G/K = \{ gK, g\in G\}
\]
has a Hausdorff Banach manifold structure such that
the quotient map $\pi(g)=gK$ is a submersion.
\end{definition}

This is guaranteed for instance, if $K$ is a split Banach--Lie subgroup (i.e the Lie algebra $\mathfrak k$ is closed in $\mathfrak g$ and has a closed complement) but we do not require it.
Since $K=\pi^{-1}(\pi(K))$ is a closed subgroup of $G$, and since we are imposing  $\pi$ to be a submersion, then in fact $K$ must be embedded in $G$, see \cite[Corollary 4.3]{aneeb} (however, it may not be split as mentioned before, see Example \ref{ex:compact}).

\medskip

\begin{notation} Let $\pi:G\to G/K$ be a homogeneous space. 
 \begin{itemize}
    \item The base point in $G/K$ will be denoted as $p_0=\pi(K)$.
    \item The action of $G$ will be denoted as $\alpha:G\times G/K\to G/K$ i.e. $\alpha(g,p)=\pi(gh)$ for $p=\pi(h)\in G/K$. 
    \item For a fixed $g\in G$, we denote by $\alpha_g\in \Aut(G/K)$ the mapping $\alpha_g(p)=\alpha(g,p)$. Similarly, for a fixed $p\in G/K$ we denote by $\alpha^p$ the mapping $\alpha^p(g)=\alpha(g,p)$. 
    \item The differential of $\alpha_g$  will be denoted by $(\alpha_g)_*:T(G/K)\to T(G/K)$. The point at which it is evaluated will be indicated as long as it is relevant or necessary. The same considerations will apply to $(\alpha^p)_*: TG \to T(G/K)$.
\end{itemize}
\end{notation}
\medskip

 The following lemmas and remarks collect the trivial (but useful for our purposes) relations among the differentials of the various maps and vector fields on $G/K$.

\medskip

 \begin{lemma}\label{derive}Let $g\in G$, $v\in\mathfrak g$. Then
 \begin{enumerate} 
     \item  $(\alpha_g)_{*}\pi_{*1}=\pi_{*g}L_g$ or equivalently $ \pi_{*g}=(\alpha_g)_*\pi_{*1}L_g^{-1}$.
\item If $p=\pi(g)\in G/K$ then $(\alpha^p)_{*1}=\pi_{*g}R_g=(\alpha_g)_*\pi_{*1}\Ad_g^{-1}$.
\item For any $h\in G$, the action property for derivative of $\alpha$ reads
\begin{equation}\label{noaction}
(\alpha_h)_{*}(\alpha_g)_{*} = (\alpha_{hg})_*.
\end{equation}
It also implies that $(\alpha_g)_*$ is a diffeomorphism of $T(G/K)$.
\item For any $k\in K$ we have $(\alpha_k)_* \pi_{*1}=\pi_{*1}\Ad_k$.
\item The kernel of the differential $\pi_{*1}:\mathfrak g \to T_{p_0}(G/K)\cong \mathfrak g/\mathfrak k$ is equal to~$\mathfrak k$.
\item $\pi_{*g}L_gv=\pi_{*h}L_h w$ iff there exists $k\in K$ such that $h=gk$ and $v-\Ad_kw\in \mathfrak k$.
  \end{enumerate}
 \end{lemma}

\begin{proof}
Differentiating the equality $\alpha_g\circ\pi = \pi\circ l_g$ at the identity $1\in G$, by chain rule we get the first claim. 

Similarly for $p=\pi(g)$ differentiating $\alpha^p=\pi\circ r_g$ at the identity and using the previous identity we get
\[ 
\pi_{*g} R_g = (\alpha_g)_*\pi_{*1}L_g^{-1} R_g = (\alpha_g)_*\pi_{*1}\Ad_g^{-1}, \]
which proves the second claim. 

The third claim follows from differentiating the property $\alpha_g\alpha_h=\alpha_{gh}$ and applying chain rule.

For the fourth claim, we write $\pi(g)=\pi(gk)=\pi(r_kg)$ and we differentiate with respect to $g$ at $g=1$ to get
\[ \pi_{*1} = \pi_{*k} R_k.\]
Thus using the first identity we get:
\[ (\alpha_k)_* \pi_{*1} = \pi_{*k} L_k = \pi_{*1} R_k^{-1} L_k = \pi_{*1} \Ad_k.\]

The fifth assertion follows from
\[
\frac{d}{dt}\,\pi(e^{tv})=(\alpha_{e^{tv}})_{*p_0}\pi_{*1}v=0.
\]
Namely if $v\in\ker \pi_{*1}$, then $\pi(e^{tv})=\pi(1)=p_0$. Thus $\{e^{tv}\}_{t\in\R}\subset K$ what implies $v\in\mathfrak k$.

The sixth assertion is immediate from the fact that the base point must be the same (hence $h=gk$) and then by the previous identities
$$
(\alpha_g)_*\pi_{*1}v=\pi_{*g}L_gv=\pi_{*h}L_hw=(\alpha_h)_*\pi_{*1}w=(\alpha_{gk})_*\pi_{*1}w=(\alpha_g)_*\pi_{*1}\Ad_kw.
$$
Since $(\alpha_g)_*$ is an isomorphism and $\ker\pi_{*1}=\mathfrak k$, the conclusion follows.
\end{proof}

\begin{definition}[Projected vector fields in $G/K$]\label{invf} The right-invariant vector fields $X^v(g)=R_gv$ on the Lie group $G$ can be pushed down to $G/K$ as follows: for $p=\pi(g)$ consider
\begin{equation}\label{Xtilda}
\widetilde{X^v}(p)=(\alpha^p)_{*1}v=\pi_{*g}R_gv=(\alpha_g)_* \pi_{*1}\Ad_g^{-1}v\in T_pG/K,
\end{equation}
where the second and third equalities come from Lemma \ref{derive}(2). Since $p\mapsto (\alpha^p)_{*1}v=\alpha_{*(1,p)}(v,0)$ depends smoothly on $p$, this defines a vector field in $G/K$.
\end{definition}

\medskip

\begin{remark}\label{rem:alpha-rel} If $p=\pi(h)$ then by Lemma \ref{derive}(3), we have
$$
(\alpha_g)_{*p}\widetilde{X^v}(p)=(\alpha_{gh})_{*p_0}\pi_{*1}\Ad_h^{-1} v,
$$
while 
$$
\widetilde{X^v}(\alpha_g(p))=\widetilde{X^v}(\pi(gh))=(\alpha_{gh})_*\pi_{*1}\Ad_{gh}^{-1}v.
$$
Since $(\alpha_{gh})_*$ is an isomorphism, the equality $(\alpha_g)_{*p}\widetilde{X^v}(p)=\widetilde{X^v}(\alpha_g(p))$ can only happen if 
$$
\Ad_h^{-1}(\Ad_g^{-1}v-v)\in \ker \pi_{*1} = \mathfrak k.
$$
So the vector fields $\widetilde{X^v}$ \emph{are not $\alpha_g$-related with themselves} in general, unlike left invariant vector fields on a Lie group. 
\end{remark}

\medskip

The following computation of Lie brackets of vector fields will be needed later:
\begin{proposition}\label{alpharel}
Fix $h\in G$ and $v\in \mathfrak g$. The projected vector field  $\widetilde{X^{\Ad_h v}}$ is the unique vector field in $G/K$ which is $\alpha_h$-related to the projected vector field $\widetilde{X^v}$. In other words
\begin{equation}\label{tildaX_and_Ad}
(\alpha_h)_{*}\widetilde{X^v}\alpha_h^{-1} = \widetilde{X^{\Ad_h v}}.   
\end{equation}
\end{proposition}
\begin{proof}Let $p=\pi(g)$, then we have $\alpha_h^{-1}(p) = \pi(h^{-1}g)$ and by equation \eqref{Xtilda} and Lemma~\ref{derive}(2) we get
\begin{align*}
(\alpha_h)_{*}\widetilde{X^v}\alpha_h^{-1} (p) &= 
(\alpha_h)_{*}(\alpha_{h^{-1}g})_* \pi_{*1}\Ad_{h^{-1}g}^{-1}v  = (\alpha_g)_{*} \pi_{*1}\Ad_{g^{-1}h}v \\ & = (\alpha_g)_{*} \pi_{*1}\Ad_{g}^{-1} \Ad_h v =\pi_{*g}R_g\Ad_hv = \widetilde{X^{\Ad_h v}}(p).\qedhere
\end{align*}
\end{proof}

\medskip

We now take a look at Lie brackets of vector fields defined on $G/K$:

\begin{proposition}\label{prop:pirel}
The fields $X^v$ and $\widetilde{X^v}$ are $\pi$-related, and in particular for $v,w\in\mathfrak g$ we have
$$
[\widetilde{X^v},\widetilde{X^w}]=-\widetilde{X^{[v,w]}},
$$
where on the left we have the Lie bracket of vector fields on the manifold $G/K$, and on the right $[v,w]=\ad_v w$ is the Lie bracket in $\mathfrak g$. 
\end{proposition}
\begin{proof}
Since $\widetilde X^v(\pi(g))=\pi_{*g} X^v(g)$ by definition, the fields are $\pi$-related. Since the Lie bracket of right invariant vector fields in $G$ is again right invariant with $[X^v,X^w]=-X^{[v,w]}$, the assertion of the lemma follows by means of \eqref{lieb}. Note the difference in signs since the Lie bracket in $\mathfrak g$ is defined as usual by left invariant vector fields, but we employ here right invariant vector fields.
\end{proof}

\begin{remark}[The problem with left-invariant fields]
    It is not possible to push down to the quotient space $G/K$ the left-invariant vector fields $g\mapsto L_g v$ on $G$. A natural candidate $X(p)=\pi_{*g}L_gv$ for $p=\pi(g)$ doesn't make sense in the quotient tangent bundle as it depends on the representative $g$ in $G$. More precisely, independence on the choice of $g$ would require $\pi_{*k}L_kv=\pi_{*1}v$ for any $k\in K$. However then
$$
\pi_{*k}L_kv=(\alpha_k)_* \pi_{*1}v=\pi_{*1}\Ad_kv
$$
by Lemma \ref{derive}(4). This would imply  $\Ad_k v-v\in \mathfrak k$ for all $k\in K$ by Lemma \ref{derive}(5), and, after differentiating with respect to $k$ at $k=1$,  also that $[z,v]\in\mathfrak k$ for any $z\in \mathfrak k$. This is not satisfied for a general $v\in \mathfrak g$. It is satisfied for $v \in \mathfrak k$ but then we obtain the zero vector field in $G/K$.
\end{remark}

\medskip

\subsection{Homogeneous vector bundle maps and admissible operators}

In this section we discuss homogeneous vector bundle maps acting in the tangent bundle $T(G/K)$, and then we discuss the ones which come from  a linear map defined ``upstairs'' in $\mathfrak g$. All vector bundle maps under considerations in this paper are understood to be covering identity, i.e. mapping each fiber to the same fiber.

\begin{definition}\label{hls}A smooth vector bundle map $\N:T(G/K)\to T(G/K)$ is called \emph{homogeneous} if it is invariant for the action by the automorphisms~$\alpha_g$:
\[
    (\alpha_g)_{*}\, \N_p=\N_{\alpha_g(p)}(\alpha_g)_{*} \qquad \textit{ for any } \; p\in G/K \; \textit{ and any  } \; g\in G.
\]
\end{definition} 
By homogeneity, any such map comes from some $\N_{p_0}\in \mathcal B(T_{p_0}(G/K))$ at the base point $p_0$. 

We now look at the situation from the side of the Lie algebra $\mathfrak g$ of the group $G$:

\begin{definition}\label{homj}We will consider the following \emph{admissible}  linear bounded operators on $\mathfrak g$:
\begin{equation}
\lgk=\{\I\in \mathcal B(\mathfrak g):  \, \I\mathfrak k\subset \mathfrak k,\; \Ran(\Ad_k \I-\I\Ad_k)\subset \mathfrak k \; \forall \, k\in K \}.
\end{equation}
\end{definition}

\begin{remark}\label{more_general}
    If $K$ is split in $G$, i.e. if $\mathfrak k$ has a closed complement $\mathfrak m$ in $\mathfrak g$, $\mathfrak g=\mathfrak k\oplus \mathfrak m$,  then one can consider the following subset of \emph{admissible}  linear bounded operators on $\mathfrak g$: 
    \[
\mathcal{B}(G, K) =\{\I\in \mathcal B(\mathfrak g):  \, \mathfrak k\subset \ker \I;\; \I\mathfrak{m} \subset \mathfrak{m};\; \Ran(\Ad_k \I-\I\Ad_k)\subset \mathfrak k \; \forall \, k\in K \}.
\]  
  Note that  $\mathcal{B}(G, K)$ is strictly contained in  $\lgk$ (see also remark~\ref{remsupl} in the case of almost complex structures).
\end{remark}

\begin{remark}\label{connected}
Note that if $\Ad_k\I - \I\Ad_k$ takes values in $\mathfrak k$ for all $k\in K$, then writing $k=e^{tz}$ with $z\in \mathfrak k$ and differentiating $e^{t\ad_ z}\I v-\I e^{t\ad_z}v\in \mathfrak k$ we see that $[z,\I v]-\I[z,v]\in\mathfrak k$ for all $z\in\mathfrak k$ and all $v\in \mathfrak g$, i.e. 
\begin{equation}\label{derivcom}
\Ran(\I\circ \ad_z - \ad_z \circ \I) \subset \mathfrak k \qquad \textrm{for all}\qquad z\in \mathfrak k.
\end{equation} 
If $K$ is connected, then it is generated by a neighborhood of the unit element and equation~\eqref{derivcom} implies $\Ran(\Ad_k\I -\I\Ad_k)\subset \mathfrak k$ for all $k\in K$. More precisely, let $V$ be an open neighborhood of the zero element in $\mathfrak{k}$ such that the restriction of the exponential map to $V$ is a diffeomorphism and consider $W = \exp V\cap \left(\exp V\right)^{-1}$. Then the subgroup $H = \bigcup_{n\geq 1} W^n$ is open in $K$. Since $K$ is the union of equivalence classes modulo $H$, $K = \bigcup_{k\in K}kH$, $H$ is the complement of the open set $\bigcup_{k\notin H} kH$, hence $H$ is also closed in $K$.  Therefore $H$  equals $K$ if the latter is connected. For each $k\in K$ we can then write $\Ad_k=\prod_i \Ad_{e^{z_i}}=\prod_i e^{\ad_{z_i}}$ for a finite number of  $z_i\in\mathfrak k$ therefore in this case \eqref{derivcom} is equivalent to $\Ran(\Ad_K \I-\I\Ad_K)\subset \mathfrak k$.
\end{remark}

\begin{definition}\label{def:homog_struct}
The \emph{homogeneous vector bundle map induced by} $\I\in \lgk$ is the smooth vector bundle map $\N:T(G/K) \to T(G/K)$ given at each $p=\pi(g)\in G/K$ by
\begin{equation}\label{Np}
\N_p = (\alpha_g)_* \N_{p_0} \, (\alpha_g)_*^{-1},
\end{equation}
where $\N_{p_0}: T_{p_0}G/K \rightarrow T_{p_0}G/K$ with $p_0 = \pi(K)$, is defined as 
\begin{equation}\label{Np0}
\N_{p_0}\pi_{*1}v:=\pi_{*1}\I v
\end{equation}
for $v\in\mathfrak g$.
\end{definition} 

\medskip

\begin{proposition}\label{prop:descend}
For any $\I\in \lgk$, the map $\N:T(G/K)\to T(G/K)$ in the previous definition is a well-defined homogeneous vector bundle map in $G/K$.
\end{proposition}
\begin{proof} First let us show that $\N_{p_0}: T_{p_0}G/K \rightarrow T_{p_0}G/K$ is well-defined by \eqref{Np0}.
Since $\pi$ is a submersion, any tangent vector in $T_{p_0}G/K$, can be written as $\pi_{*1}X$, for some $X\in \mathfrak g$. We need to show that the value of $\N_{p_0} \pi_{*1}X \in T_{p_0}G/K$ does not depend on the representative $X\in \mathfrak g$. Consider $X_1, X_2 \in \mathfrak g$ such that $\pi_{*1}(X_1) = \pi_{*1}(X_2)$, i.e. $X_1-X_2 \in \mathfrak {k}$. We have $\pi_{*1} \I(X_1-X_2) = 0$ since $\I\mathfrak k \subset \mathfrak k$. By linearity, we therefore have
\[\pi_{*1} \I(X_1) = \pi_{*1} \I(X_2).\]
Since $\pi_{*1}$ is a bounded surjection, the norm in $T_{p_0}(G/K)\simeq \mathfrak g/\mathfrak k$ must be equivalent to the quotient norm $\|\cdot\|_{quot}$ by the open mapping theorem. To show that $\N_{p_0}$  is bounded for the quotient norm in $T_{p_0}(G/K)$, note that
\[
\|\N_{p_0}[v]\|_{quot}=\inf_{z\in\mathfrak k}\|\I v-z\|\le \inf_{z\in \mathfrak k}\|\I v-\I z\|\le \|\I\| \inf_{z\in \mathfrak k}\|v-z\|=\|\I\| \|v\|_{quot}.
\]
Now we take an arbitrary vector $X_p\in T_p(G/K)$ and we write it as $X_p=\pi_{*g}L_gv=\pi_{*h}L_hw$ for some $v, w \in \mathfrak g$. Using Lemma \ref{derive}(6) it follows that $h = gk$ and $v - \Ad_k w \in \mathfrak k$, hence
\[ \pi_{*1}v = \pi_{*1}\Ad_k w. \]
Comparing values of $\N$ on those two presentations of $X_p$ we get on one hand
\begin{align*}
(\alpha_g)_* \N_{p_0} (\alpha_g)_*^{-1}X_p & =(\alpha_g)_* \N_{p_0}(\alpha_g)_*^{-1}\pi_{*g}L_g v=(\alpha_g)_* \N_{p_0}\pi_{*1}v \\
&= (\alpha_g)_* \N_{p_0}\pi_{*1}\Ad_k w= (\alpha_g)_* \pi_{*1} \I \Ad_k w \\ &= (\alpha_g)_* \pi_{*1}\Ad_k \I w
\end{align*}
since $\I$ commutes with $\Ad_k$ modulo $\ker\pi_{*1}=\mathfrak k$. On the other hand
\begin{align*}
(\alpha_h)_* \N_{p_0} (\alpha_h)_*^{-1}X_p &= (\alpha_h)_* \N_{p_0} (\alpha_h)_*^{-1}\pi_{*h}L_hw =(\alpha_h)_* \N_{p_0} \pi_{*1}w\\
&=(\alpha_h)_* \pi_{*1}\I w  =(\alpha_g)_* (\alpha_k)_* \pi_{*1}\I w =(\alpha_g)_* \pi_{*1}\Ad_k \I w
\end{align*}
by means of Lemma \ref{derive}(4). This proves that $\N$ is well-defined by \eqref{Np}. By construction $\N$ is a vector bundle map, and since it is given by the composition of smooth maps, it is smooth. The $\alpha_g$-invariance is a consequence of the definition $\N_p=(\alpha_g)_*\I(\alpha_g)_*^{-1}$,  the definition of $(\alpha_g)_*$, and the chain rule for $\alpha$.
\end{proof}

\begin{remark}\label{lifting}
    In general, not every linear morphism of $T_{p_0}(G/K)$ comes from a linear operator on $\mathfrak g$. This is related to the so called quotient lifting property of Banach spaces, see e.g. \cite{lindenstrauss,harmand}.
\end{remark}

\begin{remark}
Consider the normal bundle over $G$ with fiber at $g\in G$ given by $$\operatorname{Nor}_{g} = T_{g}G/\ker \pi_{*g}.$$  Note that $\ker \pi_{*g}$ is the vector space generated by the infinitesimal action of $K$ on $g\in G$. In particular,
$R_k \ker \pi_{*g} = \ker \pi_{*gk}$, hence the right action of $K$ on $TG$ induces a right action of $K$ on the normal bundle $\operatorname{Nor}$. Since $\pi$ is a submersion, for any $g\in G$ such that $\pi(g) = p$, the quotient Banach space $T_{g}G/\ker \pi_{*g}$ is isomorphic to $T_{p}(G/K)$. Moreover $\pi(g) = \pi(h) = p$ if and only if $h = gk$ for some $k\in K$. It follows that the tangent bundle $T(G/K)$ of the homogeneous space $G/K$ is canonically isomorphic to the quotient of the normal bundle $\operatorname{Nor}$ by the right action of $K$ (see \cite{ciuclea2023shape} for applications of this construction). 
A smooth vector bundle map $\N:T(G/K)\to T(G/K)$ is therefore equivalent to a smooth vector bundle map $\hat{\N}:\operatorname{Nor}/K\to \operatorname{Nor}/K$.
If $\pi: G\rightarrow G/K$ admits a globally defined section $s: G/K\rightarrow G$, then $\hat{\N}$ can be lifted to a $K$-invariant vector bundle map $\tilde{\N}: \operatorname{Nor}\to \operatorname{Nor}$, by defining $\tilde{\N}$ on the range of the section $s$ as $$\tilde{\N}_{s(\pi(g))}(X_{s(\pi(g))}) = \hat{\N}_{\pi(g)}\pi_{*g}(X_{s(\pi(g))}),$$ for $X_{s(\pi(g))}\in 
 \operatorname{Nor}_{s(\pi(g))}  = T_{s(\pi(g))}G/\ker \pi_{*s(\pi(g))}$ and extending it by $K$-invariance to the whole fiber bundle $\operatorname{Nor}$ over $G$. As mentioned in previous remark, not every vector bundle map $\tilde{\N}:\operatorname{Nor}\to \operatorname{Nor}$ can be lifted to a vector bundle map on $TG$. 

\end{remark}

\medskip

\begin{remark}[Projected vector fields and the homogeneous vector bundle map]\label{notrel}
The vector bundle maps can be seen also as maps on vector fields. Let's apply the map $\N:T(G/K)\to T(G/K)$ defined by means of $\I\in \mathcal{A}(G,K)$ as in Definition \ref{def:homog_struct} to the projected vector fields on $G/K$. If we compute $\N\widetilde{X^v}$, $v\in\mathfrak g$, we note that it differs from $\widetilde{X^{\I v}}$ in general. It can be seen as follows. For $p=\pi(g)$, by definition \eqref{Xtilda},  we have
\[\widetilde{X^{\I v}}(p)=(\alpha_g)_*\pi_{*1}\Ad_g^{-1}\I v,\] 
while
\begin{align*}
\N_p\widetilde{X^v}(p) & =(\alpha_g)_*\N_{p_0}(\alpha_g)_*^{-1}\widetilde{X^v}(p)=(\alpha_g)_*\N_{p_0}(\alpha_g)_*^{-1}(\alpha_g)_*\pi_{*1}\Ad_g^{-1}v\\
& =(\alpha_g)_*\N_{p_0}\pi_{*1}\Ad_g^{-1}v = (\alpha_g)_* \pi_{*1}\I\Ad_g^{-1}v.
\end{align*}
Thus for them to be equal, one must have
$$
\Ad_g^{-1}(\I v)-\I(\Ad_g^{-1}v)\in \mathfrak k,
$$
for all $g\in G$ (and not only $g \in K$), which is usually not the case. Compare with Remark \ref{split} below.

\end{remark}

\medskip

We end this section with a simple example that illustrates all the definitions and possible pitfalls: we consider the unit sphere $\mathbb{S}^2$, which can be seen as the homogeneous space $SO(3)/SO(2)$.

\subsection{Example: the sphere $\mathbb{S}^2$}\label{thesphere}
 Identify the sphere $\mathbb{S}^2$ embedded in $\R^3$ with the homogeneous space $SO(3)/SO(2)$ in such a way that the base point corresponds to the north pole $p_0 = (0,0,1)^T$. The action of $G=SO(3)$ (and $\mathfrak g=\mathfrak{so}(3)$ as well) is  given by left matrix multiplication
\[
 \begin{array}{llll}
     \alpha: &G\times \mathbb{S}^2& \longrightarrow & \mathbb{S}^2\\
     & (g, hp_0) & \longmapsto & ghp_0.
 \end{array}
\]
For each $g\in G$ the diffeomorphism $\alpha_g$ of $\mathbb{S}^2$ is then
\[
\alpha_g(hp_0) := ghp_0.
\] 
The sphere  in this picture is the orbit of the vector $p_0$ with respect to the transitive action of $G=SO(3)$ 
\[ \mathbb{S}^2 = G \cdot p_0 =  \{ g p_0\;|\; g\in G\}\]
with stabilizer $K\cong SO(2)$ given by rotations around the $z$-axis. Consider the following basis in $\mathfrak{so}(3)$:
\[
k_0=\begin{pmatrix}
    0 & 1 & 0 \\
    -1 & 0 & 0 \\
    0 & 0 & 0
\end{pmatrix},\quad 
e_1=\begin{pmatrix}
    0 & 0 & 1 \\
    0 & 0 & 0 \\
    -1 & 0 & 0
\end{pmatrix}, \quad 
e_2=\begin{pmatrix}
    0 & 0 & 0 \\
    0 & 0 & 1 \\
    0 & -1 & 0
\end{pmatrix}.\]
Note that $k_0$ spans the Lie algebra $\mathfrak k \cong \mathfrak{so}(2)$ and that $\{e_1p_0, e_2p_0\}$ spans the tangent space to the sphere $\mathbb{S}^2$ at $p_0$. We have the following commutation relations 
\[
 [k_0,e_1]=-e_2, \quad [k_0,e_2]=e_1, \quad [e_1,e_2]=-k_0.
\]

Right-invariant vector fields on $G = SO(3)$ descend to vector fields on the sphere $\mathbb{S}^2$ according to Definition \ref{invf}. Let $v$ be an element in the Lie algebra $\mathfrak{so}(3)$. The right invariant vector field $X^v$ on $G$ whose value at the identity equals $v$ takes the value $vg$ at an arbitrary point $g\in G$. 
It descends to the following projected vector field on the sphere $\mathbb{S}^2$
\[\widetilde{X^v}(p)=(\alpha^p)_{*1}v= vp,\qquad\textrm{for } p\in \mathbb{S}^2.\]

The diffeomorphism $\alpha_g:\mathbb{S}^2\to \mathbb{S}^2$ transforms the vector field $\widetilde{X^v}$ in the following manner (see equation~\eqref{tildaX_and_Ad} in Proposition \ref{alpharel}):
\[ 
gvp=g\widetilde{X^v}(p) = (\alpha_g)_*(\widetilde{X^v})(p) = (\alpha_g)_*(\widetilde{X^v})((\alpha_g)^{-1}(\alpha_g p)) = \widetilde{X^{\Ad_g v}}(gp),
\]
while $\widetilde{X^v}(gp)=vgp$.  In particular, in general $g\widetilde{X^v}(p) \neq \widetilde{X^v}(gp)$ i.e. the projected vector fields $\widetilde{X^v}$ are not $\alpha_g$-related with themselves, see Remark \ref{rem:alpha-rel}. This is shown explicitly here for particular values of $v$ and $g$: let $v=e_1$ and $g=\begin{pmatrix}
    1 & 0 & 0 \\
    0 & -1& 0 \\
    0 & 0 & -1\end{pmatrix}\in SO(3)$. Then 
\[
g \widetilde{X^v}(p_0) =ge_1p_0= g\,  (1,0,0)^T=(1,0,0)^T,
\]
while 
\[
\widetilde{X^v}(gp_0)=  e_1(-p_0)=-e_1p_0=(-1,0,0)^T,
\]
which are different.

\bigskip

Now consider the operator $\I = \ad_{k_0}$ on the Lie algebra $\mathfrak{so}(3)$,
i.e.
\[
\I k_0 = 0,\qquad \I e_1 = -e_2, \qquad \I e_2=e_1.
\]
It preserves $\mathfrak k = \mathbb{R} k_0$. In order to show that it is admissible, we have to verify the claim $\Ran(\Ad_k \I - \I \Ad_k) \subset  \mathfrak k$ for $k\in K$ (see Proposition \ref{prop:descend}). Since $K$ is a connected group, by Remark \ref{connected} it is enough to verify it on the level of the Lie algebra, which is trivial since 
 $\I=\ad_{k_0}$. Hence $\I \in \mathcal{A}(SO(3),K)$ i.e. $\I$ is admissible. Therefore $\N$ descends to a linear operator $\N_{p_0}$ on the tangent space to the sphere at $p_0$ (see Proposition \ref{prop:descend}), defined by
\begin{equation}\label{Iquot}
\N_{p_0}(e_ip_0):=(\I e_i)p_0\qquad \textrm{ in }\{p_0\}^{\perp}=T_{p_0}\mathbb{S}^2,
\end{equation}
and we have the property 
\begin{equation}\label{Np0_so3}
\N_{p_0}(vp_0) =  (\I v)(p_0),\qquad v\in \mathfrak{so}(3).
\end{equation}

We define now the vector bundle map $\N$ on $\mathbb{S}^2$ using the homogeneous action of the group, following Definition \ref{def:homog_struct}:
\[ 
\N_{gp_0} z : = g \N_{p_0} (g^{-1} z), \qquad z\in T_{gp_0}\mathbb{S}^2\subset \R^3,
\]
 hence $g^{-1}z\in T_{p_0}\mathbb{S}^2=\spn\{e_1p_0,e_2p_0\}$ and $\I(g^{-1}z)$ is defined by means of \eqref{Iquot}. This does not depend on the choice of a particular $g$ which we use to obtain $p=gp_0$ because $\I$ is admissible.

Let us present the situation discussed in Remark \ref{notrel}. We have
\begin{align*}    
(\N\widetilde{X^v})(gp_0) & =
\N_{gp_0} \widetilde{X^v}(gp_0) = g \N_{p_0} g^{-1} \widetilde{X^v}(gp_0) = g\N_{p_0} (g^{-1} vgp_0) \\
&= g\I(g^{-1}vg)p_0 = g\I(\Ad_{g}^{-1}v)p_0.
\end{align*}
On the other hand $\widetilde{X^{\I v}} (gp_0)=(\I v)gp_0$. In the general case, those two expressions are far from being equal unless $g\in K$: let us show an explicit example. Consider 
\[ g=\begin{pmatrix}
    1 & 0 & 0 \\
    0 & \cos\theta & -\sin \theta\\
    0 & \sin\theta & \cos\theta
\end{pmatrix}\in SO(3)\] 
and $v=e_2$.
Then $gp_0 = (0,-\sin\theta,\cos\theta)^T$ and $\I v = e_1$. We compute now
\[\widetilde{X^{\I v}}(gp_0) = e_1gp_0 = (\cos\theta, 0 , 0)^T.\]
On the other hand
\[ g^{-1}e_2g = e_2 \xrightarrow{\I} e_1, \]
which implies
\[ (\N\widetilde{X^v})(gp_0) = g\I (g^{-1} e_2 g)p_0 = g e_1 p_0 = g (1,0,0)^T=(1,0,0)^T.\]
In conclusion $\N\widetilde{X^v} \neq \widetilde{X^{\I v}}$. 

\section{The Nijenhuis torsion of a vector bundle map $\N$}\label{section:ntorsion}

\begin{definition}\label{torJ} Let $\mathcal M$ be any smooth manifold and let $\N:T\mathcal M\to T\mathcal M$ be a smooth vector bundle map. The \emph{Nijenhuis torsion} of $\N$ is defined as 
\[
\Omega_{\N}(X,Y)=\N [\N X,Y]+\N [X,\N Y]-[\N X,\N Y]-\N^{\; 2}[X,Y]
\]
for $X,Y$ vector fields in $\mathcal M$. We say that $\N$ is a \emph{Nijenhuis operator} in $\mathcal M$ if its torsion vanishes.
\end{definition}

Note that $\Omega_{\N}$ is anti-symmetric.
The following is well-known for finite dimensional manifolds, see also \cite[Lemma 2]{beltita05integrability} for almost complex structures in the Banach setting (i.e. vector bundle maps such that $\N^2=-1$). Let $E$ denote the Banach space modelling the manifold $\mathcal M$.

\begin{theorem}\label{onlyp} The Nijenhuis torsion of $\N$ at $p\in \mathcal M$ depends only on the values of the vector fields at the point $p$, i.e. $\Omega_{\N}$ is  a tensor. In any manifold chart $(U,\varphi)$, using the local expression of   $\N:\varphi(U)\subset E\to \mathcal B(E)$  and the local expressions of the vector fields $X,Y:\varphi(U)\to E$, one has
\begin{align*}
\Omega_{\N}(X,Y)_p & =\N_p\big( (\N_{*p}X_p)(Y_p)- (\N_{*p}Y_p)(X_p)\big)\\
& \quad +\big(\N_{*p}(\N_pY_p)\big)(X_p)- \big(\N_{*p}(\N_pX_p)\big)(Y_p).
\end{align*} 
\end{theorem}
\begin{proof}
First we note that for the local expressions of $X,Y$, we can compute their Lie bracket as  $[X,Y]_p=Y_{*p}(X_p)-X_{*p}(Y_p)$. On the other hand, for the local expression of $\N$ we have $\N_{*p}:E \to \mathcal B(E)$, and since $(\N,X)\mapsto \N X$ is bilinear, we can use the Leibniz rule to compute 
\[
(\N X)_{*p}v=(\N_{*p}v)(X_p)+\N_p( X_{*p}v) \qquad \forall v\,\in E.
\]
Then 
\[
[\N X,Y]_p= Y_{*p}( \N_p X_p)- (\N_{*p}Y_p)(X_p)-\N_p( X_{*p}Y_p) 
\]
while 
\[
[X,\N Y]_p=(\N_{*p}X_p)(Y_p)+\N_p( Y_{*p}X_p)-X_{*p}(\N_p Y_p) 
\]
and
\begin{align*}
[\N X,\N Y]_p   = & \big(\N_{*p}(\N_p X_p)\big)(Y_p) + \N_p\big( Y_{*p}(\N_p X_p) \big)  \\
& \;  -\big(\N_{*p}(\N_p Y_p)\big)(X_p) - \N_p\big( X_{*p}(\N_p Y_p) \big).
\end{align*}
It follows that
\begin{align*}
\Omega_{\N}(X,Y)&=\N [\N X,Y]+\N [X,\N Y]-[\N X,\N Y]-\N^{\; 2}[X,Y]\\
&= \N_p \big(Y_{*p}( \N_p X_p)\big) - \N_p \big((\N_{*p}Y_p)(X_p)\big) + \N_p \big((\N_{*p}X_p)(Y_p)\big) - \N_p \big(X_{*p}(\N_p Y_p) \big)\\
&-\big(\N_{*p}(\N_p X_p)\big)(Y_p) - \N_p\big( Y_{*p}(\N_p X_p) \big) + \big(\N_{*p}(\N_p Y_p)\big)(X_p) + \N_p\big( X_{*p}(\N_p Y_p) \big)\\
\end{align*}
The stated formula follows after cancelling out terms.
\end{proof}

We now return to the homogeneous structures to give a local/global expression of the torsion of $\N$.

\begin{remark}[Exponential map of $G$]Let $V\subset \mathfrak g$ be a $0$-neighbourhood such that $\exp|_V:V\to U=\exp(V)$ is a diffeomorphism. Recall the formula for the differential of the exponential map
\begin{equation}\label{diff_exp}
\exp_{*z}x=L_{e^z}F(\ad_z)x=R_{e^z}f(\ad_z)x,
\end{equation}
where $F(\lambda)=(1-e^{-\lambda})/\lambda$ and $f(\lambda)=e^{\lambda}F(\lambda)$. The proof of these formulas for finite dimensional Lie groups can be found in Helgason's book  \cite[Chapter IV, Theorem 4.1]{helgason78}; for a proof adapted to the Banach setting see for instance \cite[Appendix A]{tumpach-larotonda}.

From now on we denote by $L_g$ the differential $(\ell_g)_{*h}$ at any $h\in G$, for short. We note that
\[
F(\lambda)=1-\tfrac{1}{2}\lambda+O(\lambda^2)\quad \textrm{ while } \quad f(\lambda)=1+\tfrac{1}{2}\lambda+O(\lambda^2).
\]
\end{remark}

\medskip

\begin{lemma}
For an admissible operator $\I\in \lgk$ and $v\in\mathfrak g$ consider the right invariant vector fields $X^v$ (Definition \ref{invf}) and the vector fields $X^{v,\I}$ defined as 
\[X^{v,\I}(g)=L_g \I\Ad_g^{-1}v.\]
Then one has
\begin{enumerate}
    \item $L_g^{-1}[X^{v,\I},X^w](g)=-\I[v_0,w_0]$.
     \item $L_g^{-1}[X^{v,\I},X^{w,\I}](g)=[ \I v_0,\I w_0]-\I[v_0,\I w_0]-\I[\I v_0,w_0]$,
     \end{enumerate}
where $v,w\in\mathfrak g$,  $v_0=\Ad_g^{-1}v$ and $w_0=\Ad_g^{-1}w$.
    \end{lemma}

\begin{proof}
We use the exponential chart $(gU,\varphi)$ around $g\in G$ to compute the Lie brackets, i.e. $\varphi: gU \rightarrow V\subset \mathfrak g$ and  $\varphi^{-1}(z)=ge^z$ for $z\in V$. Consider the local expressions $\overline{X}^{v,\I}$, $\overline{X}^v$ of $X^{v,\I}$, $X^v$ in this chart, i.e.
\[
\overline{X}^v(z)=\varphi_{*ge^z}X^v(ge^z), \quad \overline{X}^{v,\I}(z)=\varphi_{*ge^z}X^{v,\I}(ge^z),
\]
for $z\in V$. We note that  $L_{ge^z}=L_gL_{e^z}$ and $R_{ge^z}=R_{e^z}R_{e^g}$, and that $L$ commutes with $R$. Differentiating the identity $\varphi(ge^z)=z$ we obtain $\varphi_{*ge^z}(l_g)_{*e^z}\exp_{*z}=\id_{\mathfrak g}$ or equivalently using the formula \eqref{diff_exp}
\[
\varphi_{*ge^z}=F(\ad_z)^{-1}L_{e^z}^{-1}L_g^{-1}=f(\ad_z)^{-1}R_{e^z}^{-1}L_g^{-1}=f(\ad_z)^{-1}L_g^{-1}R_{e^z}^{-1}
\]
by the previous remark. Then plugging in the formula from Definition \ref{invf} we get
\[
\overline{X}^v(z)=f(\ad_z)^{-1}L_g^{-1}R_gv=f(\ad_z)^{-1} \Ad_g^{-1}v
\]
or equivalently $f(\ad_z)\overline{X}^v(z)=\Ad_g^{-1}v$. Replacing $z$ with $tz$ we see that 
\[
\overline{X}^v(tz)+ \tfrac{1}{2}[tz, \overline{X}^v(tz)]+ O(t^2)=\Ad_g^{-1}v.
\]
Then by differentiating at $t=0$ we get
\[
\overline{X}^v_{*0}(z)= -\tfrac{1}{2}[z,\overline{X}^v(0)] = -\tfrac{1}{2}[z,\Ad_g^{-1}v].
\]
Applying the same approach to $X^{v,\I}$ we get 
\[
\overline{X}^{v,\I}(z)=F(\ad_z)^{-1}\I(\Ad_{e^z}^{-1}\Ad_g^{-1}v)
\]
or equivalently $F(\ad_z)\overline{X}^{v,\I}(z)=\I(e^{-\ad_z}\Ad_g^{-1}v)$. Then analogously to the previous case we obtain
\[
\overline{X}^{v,\I}(tz)-\tfrac{1}{2}[tz,\overline{X}^{v,\I}(tz)]+O(t^2)=\I\Ad_g^{-1}v- \I[tz,\Ad_g^{-1}v]+O(t^2).
\]
Thus again differentiating at $t=0$ we arrive at
\[
\overline{X}^{v,\I}_{*0}(z)=\tfrac{1}{2}[z,\overline{X}^{v,\I}(0)]-\I[z,\Ad_g^{-1}v]=\tfrac{1}{2}[z,\I\Ad_g^{-1}v]-\I[z,\Ad_g^{-1}v].
\]
Now we compute $[\overline{X}^{v,\I},\overline{X}^w](0)= \overline{X}^w_{*0}(\overline{X}^{v,\I}_0)- \overline{X}^{v,\I}_{*0}(\overline{X}^w_0)$ which gives us 
\begin{align*}
    [\overline{X}^{v,\I},\overline{X}^w](0)=&-\tfrac{1}{2}[\I\Ad_g^{-1}v,\Ad_g^{-1}w] -\tfrac{1}{2}[\Ad_g^{-1}w,\I\Ad_g^{-1}v]+\I[\Ad_g^{-1}w,\Ad_g^{-1}v]\\
    & = \I[\Ad_g^{-1}w,\Ad_g^{-1}v]=\I[w_0,v_0].
\end{align*}
Applying $\varphi^{-1}_{*0}=L_g$ gives the first formula of the lemma. The other  Lie bracket 
\[
[ \overline{X}^{v,\I}  ,  \overline{X}^{w,\I} ](0) = \overline{X}^{w,\I}_{*0}(\overline{X}^{v,\I}_0)- \overline{X}^{v,\I}_{*0}(\overline{X}^{w,\I}_0)
\]
can be computed in a similar fashion in order to arrive at the second formula of the lemma.
\end{proof}

\medskip

\begin{lemma}\label{nihenG}Let $v,w\in\mathfrak g$ and consider the projected vector fields $\widetilde{X^v},\widetilde{X^w}$ in $G/K$ (Definition \ref{invf}). Then at $p=\pi(g)$, we have
$$
\N[\widetilde{X^v}, \N\widetilde{X^w}](p)+\N[\N\widetilde{X^v}, \widetilde{X^w}](p)=-2(\alpha_g)_*\pi_{*1}\I^2[v_0,w_0]
$$
and 
$$
[\N\widetilde{X^v}, \N\widetilde{X^w}](p)=(\alpha_g)_*\pi_{*1}\big([\I v_0,\I w_0]-\I[v_0,\I w_0]-\I[\I v_0,w_0]\big),
$$
where $v_0=\Ad_g^{-1}v$ and $w_0=\Ad_g^{-1}w$.
\end{lemma}
\begin{proof}We first note that the vector fields $\N\widetilde{X^v}$ are $\pi$-related to the vector fields $X^{v,\I}$ of the previous lemma i.e. for any $g\in G$
\begin{align*}
\pi_{*g}X^{v,\I}(g) & =\pi_{*g}L_g\I\Ad_g^{-1}v=(\alpha_g)_*\N_{p_0}\pi_{*1}\Ad_g^{-1}v\\
&=(\alpha_g)_*\N_{p_0}(\alpha_g)_*^{-1} \widetilde{X^v}(p)=\N_p\widetilde{X^v}(p)
\end{align*}
by Lemma \ref{derive}(1). 
We also recall that $\widetilde{X^v}$ are $\pi$-related with $X^v$
by Proposition \ref{prop:pirel}. Thus again by Lemma \ref{derive}(1), the previous lemma, and equality \eqref{lieb} we obtain
\begin{align*}
[\N\widetilde{X^v}, \widetilde{X^w}](p)  =& \pi_{*g}[X^{v,\I},X^w](g)=(\alpha_g)_*\pi_{*1}L_g^{-1}[X^{v,\I},X^w](g) \\
=& - (\alpha_g)_*\pi_{*1} \I[v_0,w_0].
\end{align*}
By reversing the bracket and exchanging $v,w$ we also get $
[\widetilde{X^v}, \N\widetilde{X^w}](p)=-(\alpha_g)_*\pi_{*1}\I[v_0,w_0]$. Thus summing and applying $\N_p=(\alpha_g)_* \N_{p_0}(\alpha_g)_*^{-1}$, with $\N_{p_0}\pi_{*1}v:=\pi_{*1}\I v$,  we obtain the first formula  of  the lemma. To compute the second bracket we use the previous lemma and the result is straightforward.
\end{proof}

\bigskip

\begin{theorem}\label{torsionJ} Assume $G/K$ is equipped with a homogeneous vector bundle map $\N$ induced by $\I\in\lgk$. Let $X,Y$ be vector fields on $G/K$. Let $p=\pi(g)$ and take $v,w\in \mathfrak g$ such that $X(p)=(\alpha_g)_*\pi_{*1}v$ and $Y(p)=(\alpha_g)_*\pi_{*1}w$. Then the Nijenhuis torsion of $\N$ can be expressed as
$$
\Omega_{\N}(X,Y)(p)=(\alpha_g)_*\pi_{*1}\big(\I[v,\I w]+\I[\I v,w]-[\I v,\I w]-\I^2[v,w]\big).
$$
In particular $\N$ is a Nijenhuis operator in $G/K$, i.e. $\Omega_{\N}\equiv 0$, if and only if
\begin{equation}\label{localtor}
\I[v,\I w]+\I[\I v,w]-[\I v,\I w]-\I^2[v,w]\in \mathfrak k
\end{equation}
for all $v,w\in\mathfrak g$.
\end{theorem} 
\begin{proof}
Since the value of the torsion tensor depends only on the values of the vector fields at the considered point (Theorem \ref{onlyp}), we fix $g$ and we replace $X,Y$ with the projected vector fields with speeds $\Ad_g v$ and $\Ad_g w$ respectively (i.e. $\widetilde{X}(\pi(h))=\widetilde{X^{\Ad_g v}}(\pi(h))=(\alpha_h)_* \pi_{*1}\Ad_h^{-1}\Ad_gv$ and likewise $\widetilde{Y}=\widetilde{X^{\Ad_g w}}$, as in Definition \ref{invf}). Then we compute the torsion $\Omega_{\N}$ of these two vector fields, and almost all the computations were done in the previous lemmas. We only need to add that their Lie bracket is
$$
[\widetilde X,\widetilde Y](p)=\pi_{*g} R_g [\Ad_g w, \Ad_g v]=-\pi_{*g}R_g\Ad_g[v,w]=-(\alpha_g)_* \pi_{*1}[v,w]
$$
by means of Lemma \ref{derive}(2). Therefore 
\[
\N^{\; 2}[\widetilde X,\widetilde Y](p)=-(\alpha_g)_* \pi_{*1}\I^2[v,w],
\]
which then cancels out one of the brackets in Lemma \ref{nihenG}.
\end{proof}

\begin{remark}
The expression in \eqref{localtor} is actually the value at identify of the Nijenhuis torsion of the left-invariant bundle map $TG\to TG$ defined by $\I$.
\end{remark}

\begin{corollary}\label{oneof}If either $v$ or $w$ belong to $\mathfrak k$, then \eqref{localtor} is automatically fulfilled for admissible $\I$. If there exists a linear complement $\mathfrak m$ of $\mathfrak k$, it suffices to check  \eqref{localtor} for $v,w\in \mathfrak m$.
\end{corollary} 
\begin{proof}
By the anti-symmetry in $v,w$ of \eqref{localtor}, it suffices to verify the first claim for $v\in\mathfrak k$. By equation~\eqref{derivcom} $\I$ commutes with the adjoint action by elements in $\mathfrak k$ modulo $\mathfrak k$
and we have
\[
\I [v,\I w ]= \I^2 [v,w] + k_1,
\]
for some $k_1\in\mathfrak k$, hence the first term in equation~\eqref{localtor} cancels with the fourth. Since $\I v\in\mathfrak k$ also, we have
\[
\I [\I v, w]=[\I v, \I w] + k_2
\]
for some $k_2\in\mathfrak k$,
and the second term cancels with the third one. Now if $\mathfrak g=\mathfrak m\oplus \mathfrak k$, by the bilinearity of the torsion $\Omega_\N$ and the previous claim, the second claim follows.
\end{proof}

\begin{corollary}\label{cor:ad_nij}
In the case when $\I$ is defined as $\ad_d$ for some $d\in \mathfrak g$, it is an admissible operator (see Definition~\ref{homj}) if and only if for all $v\in\mathfrak g$ and $k\in\mathfrak k$
\[ [k,d] \in \mathfrak k\] and 
\[ [v,[k,d]] \in \mathfrak k.\]
The condition \eqref{localtor} in this case simplifies to
\[ \big[ [d,v] , [d,w] \big] \in \mathfrak k\]
for all $v,w\in\mathfrak g$.
\begin{proof}
Simply expand \eqref{localtor} and apply the Jacobi identity.
\end{proof}
\end{corollary}
\begin{remark}\label{split}
If $K$ is split in $G$, i.e. $\mathfrak g=\mathfrak k\oplus \mathfrak m$, then around each $p\in G/K$ there  exists a  smooth local cross-section $\sigma:U\subset G/K\to G$ for the quotient map (i.e. $\pi\circ\sigma=\id_U$, see  \cite[Theorem 4.19]{beltita2006}). Then the proof of Theorem \ref{torsionJ} can be simplified (following Fr\"olicher \cite{frolicher55} and his proof of his Satz 2 in Section 19, for almost complex structures) by considering the local vector fields on the homogeneous space 
\[ 
\widehat v_p = \pi_{*\sigma(p)}  L_{\sigma(p)}v, \qquad p\in U,\quad v\in \mathfrak g.
\]
It is plain that $\widehat{v}$ is $\pi$-related to the restriction   of the left invariant vector field generated by $v$ to the submanifold $\sigma(U)$, 
and also that 
\[
(\N\, \widehat v)_p=\N_p\widehat v_p= \pi_{*\sigma(p)}  L_{\sigma(p)}\I v = (\widehat{\I v})_p,
\]
i.e. $\N$ exchanges the field induced by $v\in \mathfrak g$ with the one induced by $\I\!v\in \mathfrak g$.

\end{remark}

\section{Almost complex structures}\label{section:almost}

Let us recall that by almost complex structure $\J$ on a manifold $\mathcal M$ we mean a vector bundle map $\J:TM\to TM$ such that $\J^2=-1$. 
Its Nijenhuis torsion defined in Definition \ref{torJ} is
\begin{equation}\label{nitensor}
\Omega_{\J}(X,Y)={\J}\big([{\J}X,Y]+[X,\J Y]\big)-[{\J}X,{\J}Y] + [X,Y],
\end{equation}
where $X$ and $Y$ are vector fields in $\mathcal{M}$, and the bracket $[\cdot, \cdot]$ denotes the bracket of vector fields.

\begin{remark}\label{remnij}
For finite dimensional manifolds, the vanishing of this tensor is equivalent to the integrability of the almost complex structure by the Newlander--Nirenberg theorem \cite{eckmann51,newlander}. In the infinite-dimensional setting, this is not always true. An example of an infinite-dimensional smooth almost complex Banach manifold with a vanishing Nijenhuis tensor, that is not integrable, was given by Patyi in \cite{patyi2000}. However, as was shown in \cite{beltita05integrability} and in the Appendix of \cite{tumpach-phd}, for real-analytic Banach manifolds endowed with real-analytic almost complex structures, the Newlander--Nirenberg theorem reduces to the Frobenius theorem for the eigenspaces of the complex linear extension $\J^{\C}$ of $\J$ to the complex analytic extension of the tangent bundle $T\M^{\C}$ by the same argument as employed in \cite{eckmann51}. It is therefore true in this context. 

The example in \cite{patyi2000} shows that the construction of the complex analytic extension of the tangent bundle $T^{\C}\M$ may not be possible when the structure is only smooth and not real-analytic. The obstruction is the lack of certain properties of PDEs which hold in finite dimensional vector spaces, but are not available in the Banach setting. For context and better explanation of these remarks, see the proof of Malgrange in Nirenberg's lecture notes \cite[Theorem 4]{nirenberg73}. 
\end{remark}

\medskip

\begin{definition}A \emph{homogeneous almost complex structure} is a homogeneous vector bundle map $\J$ in $G/K$ (Definition \ref{hls}) with the additional requirement that  $\J^2=-1$. 
\end{definition}

In our homogeneous setting we are interested in those $\J$ that are induced by admissible linear bounded operators via Definition~\ref{def:homog_struct}.

\begin{definition}\label{admJ}
Consider the following subset of admissible operators on $\mathfrak g$
\[ \lcgk = \{ \V\in \lgk \;|\; \Ran(\V^2+1)\subset \mathfrak k\}.\]
Note that if $\J$ is induced by $\V \in \lcgk$  one has  $\J_{p_0}^2=-1$ in $T_{p_0}(G/K)$, therefore $\J^2=-1$ in the whole tangent bundle $T(G/K)$.
\end{definition}
\medskip

\begin{remark}\label{more_general_complex}
    In the case when $K$ is a split subgroup of $G$, i.e. when $\mathfrak k$ has a closed complement $\mathfrak m$ in $\mathfrak g$, $\mathfrak g=\mathfrak k\oplus \mathfrak m$,  then one can consider the following subset of $\lcgk$
    \[
\mathcal{I}_{\mathfrak m}(G, K) =\{\I\in \lcgk:  \, \mathfrak k\subset \ker \I;\; \I\mathfrak{m} \subset \mathfrak{m};\; \I_{| \mathfrak m}^2 = - 1_{|\mathfrak m} \}.
\]  
  Note that  $\mathcal{I}_{\mathfrak m}(G, K)$ is strictly contained in  $\lcgk$. By Proposition~12 in \cite{beltita05integrability}, in this complemented case, any almost complex structure on $M = G/K$ is induced by a linear map in $\mathcal{I}_{\mathfrak m}(G, K)$. The non-complemented case is more tricky and, as mentioned in Remark~\ref{lifting} relates to the quotient lifting property of Banach spaces.
\end{remark}

\begin{notation}[Complexification]
Let $\mathfrak g^{\mathbb C}=\mathfrak g\oplus  i\mathfrak g$ be the complexification of the Banach--Lie algebra $\mathfrak g$, and denote by $\mathfrak k^{\mathbb C}$ the complexification of $\mathfrak k$. Relative to the splitting $\mathfrak g^{\mathbb C}=\mathfrak g\oplus  i\mathfrak g$, the complex conjugation maps an element  $x=a+ib\in \mathfrak g$ to its complex conjugate defined by $\overline{x}=a-ib$. For the complexified Lie-bracket, it is plain that
\begin{equation}\label{conjbr}
\overline{[x,y]}=[\overline{x},\overline{y}].
\end{equation}
\end{notation}

We will denote by $\V^{\mathbb{C}}$ the complex linear extension of $\V\in \mathcal B(\mathfrak g)$  to $\mathfrak g^{\mathbb C}$, i.e. $\V^{\mathbb{C}}(a+ib)=\V a+ i \V b$. Note that $\V^{\mathbb{C}}$ is  a  bounded operator on $\mathfrak g^{\mathbb C}$, which satisfies:
\begin{itemize}
\item $\V^{\mathbb{C}}(\overline{x})=\overline{\V^{\mathbb{C}}(x)}$ for any $x\in \mathfrak g^{\mathbb C}$;
\item for $\V\in \lgk$, its complexification $\V^{\mathbb{C}}$ preserves $\mathfrak k^{\mathbb C}$ and we have
\begin{equation}\label{complex-commutation}
\V^{\mathbb{C}}[k, v] = [k, \V^{\mathbb{C}}v]
\end{equation}
for any $k\in \mathfrak{k}^{\mathbb{C}}$ and $v\in \mathfrak{g}^{\mathbb{C}}$;
\item for $\V\in \lcgk$, the following holds $\Ran\big((\V^{\mathbb{C}})^2+1\big)\subset \mathfrak k^{\mathbb C}$.
\end{itemize}

\begin{definition}\label{invs} 
Define the following  two subspaces 
\[
Z_{\pm}=\{v\in \mathfrak g^{\mathbb C}: (\V^{\mathbb{C}} \mp  i)v\in 
 \mathfrak k^{\mathbb C}\},
\]
that is $v\in Z_+$ if $\V v=iv +k$ for some $k\in \mathfrak  k^{\mathbb C}$ and likewise with $Z_-$. Note that $\overline{Z_+}=Z_-$ and that $Z_+$ and $Z_-$ are closed as preimages of the closed subalgebra $\mathfrak k^{\mathbb C}$ by a continuous map. 
\end{definition}
\medskip

\begin{theorem}\label{teob} A homogeneous almost complex structure $\J$ in $G/K$ induced by $\V\in\lcgk$ is Nijenhuis (i.e. the condition $\Omega_{\J}\equiv 0$ holds) if and only if  $[v,w]\in Z_+$ for all $v,w\in Z_+$. 
\end{theorem}
\begin{proof}
We use the formula \eqref{localtor} from Theorem \ref{torsionJ}, with the addition that for $v, w \in \mathfrak g$, 
\begin{equation}\label{range}
\V^2[v,w]+[v,w]\in\mathfrak k
\end{equation}
Let us define an anti-symmetric bilinear form $\beta$ on $\mathfrak g$ by  
\begin{equation}\label{localnj}
\beta(v,w):=\V[v,\V w]+\V[\V v,w]-[\V v,\V w] - \V^2[v,w].
\end{equation}
In this setting the vanishing of the torsion is therefore equivalent to  $\beta$  taking values in  $\mathfrak k$: 
$$\Omega_{\J}\equiv 0 \Longleftrightarrow \beta(v,w) \in \mathfrak k, \forall v,w\in \mathfrak g.$$
\begin{itemize}[leftmargin=*]
   \item Suppose that  $\Omega_{\J}$ vanishes. By complexifying the bilinear form $\beta$ defined by equation (\ref{localnj}) and using equation~\eqref{range}, we have for $v, w \in \mathfrak{g}^{\mathbb{C}}$
   \[
   \V^{\mathbb{C}}[v,\V^{\mathbb{C}} w]+\V^{\mathbb{C}} [\V^{\mathbb{C}} v,w]-[\V^{\mathbb{C}} v,\V^{\mathbb{C}} w]+[v,w]\in \mathfrak{k}^{\mathbb{C}}.
   \]
    Let us prove that if   $v,w\in Z_+$, then the bracket $[v, w]$ belongs to  $ Z_+$ as well.  For $v,w\in Z_+$, we have $\V^{\mathbb C}v=iv+k_1$ and $\V^{\mathbb C}w=iw+k_2$. Therefore 
\[
2i\V^{\mathbb C}[v,w]+\V^{\mathbb C}[v,k_2]+\V^{\mathbb C}[k_1,w]+2[v,w]-i[v,k_2]-i[k_1,w]\in\mathfrak{k}^{\mathbb{C}}.
\]
By equation~\eqref{complex-commutation}  we have
\[
\V^{\mathbb C}[v,k_2]=-\V^{\mathbb C}[k_2,v]=-[k_2,\V^{\mathbb C}v]=-[k_2,iv+k_1]=i[v,k_2]+k_3
\]
and likewise $\V^{\mathbb C}[k_1,w]=i[k_1,w]+k_4$, therefore $$2i\V^{\mathbb C}[v,w]=-2[v,w] +k_5,$$ which proves that $[v,w]\in Z_+$.

\item Now we prove the implication in reverse direction. Suppose that for $v,w\in Z_+$, the bracket $[v, w]$ belongs to $ Z_+$. Let us prove that $\beta$ takes values in~$\mathfrak k$. 

Let $v,w\in \mathfrak g$, then
\[
(\V^{\mathbb C}-i)(\V^{\mathbb C}+i)v=(\V^2+1)v\in \mathfrak k,
\]
therefore $(\V^{\mathbb C}+i)v\in Z_+$, and likewise $(\V^{\mathbb C}+i)w\in Z_+$. Then the hypothesis of the theorem tells us that
$$
\V^{\mathbb C} [(\V^{\mathbb C} +i)v,(\V^{\mathbb C} +i)w]=i[(\V^{\mathbb C} +i)v,(\V^{\mathbb C} +i)w]+k
$$
for some $k\in \mathfrak k^{\mathbb C}$. After expanding, we get that 
\begin{equation}\label{launo}
\V [\V v,\V w]+i\V [v,\V w]+i\V [\V v,w]-\V [v,w]
\end{equation}
equals to
\[
i[\V v,\V w]-[\V v,w]-[v,\V w]-i[v,w]+k.
\]
Note that by using the hypothesis, conjugating, using equation \eqref{conjbr} and the fact that $\overline{Z_+}=Z_-$, we also have $\V^{\mathbb C} [x,y]+i[x,y]\in\mathfrak k^{\mathbb C}$ for all  $x,y\in Z_-$. Since we also have $(\V^{\mathbb C} -i)v\in Z_{-}$ and  $(\V^{\mathbb C} -i)w\in Z_-$, with a similar reasoning we obtain that 
\begin{equation}\label{lados}
\V [\V v,\V w]-i\V [v,\V w]-i\V [\V v,w]-\V [v,w]
\end{equation}
equals to
\[
-i[\V v,\V w]-[\V v,w]-[v,\V w]+i[v,w]+k'.
\]
Adding equations \eqref{launo} and \eqref{lados} (and halving) and canceling out we arrive at
$$
[v,\V w]+[\V v,w]+\V [\V v,\V w]-\V [v,w]\in \mathfrak k.
$$
If we apply $\V$ we get
$$
\V [v,\V w]+\V [\V v,w]+\V^2[\V v,\V w]-\V^2[v,w]\in \mathfrak k.
$$
Finally using equation~\eqref{range}, 
 $\V^2[\V v,\V w]=-[\V v,\V w]+k_2$  for $k_2\in \mathfrak k$, hence $\beta(v,w)\in\mathfrak k$.
\end{itemize}
\end{proof}

Combining Definitions \ref{homs}, \ref{homj}, \ref{invs}, Theorem \ref{teob} and the Newlander--Nirenberg theorem in the Banach context (\cite[Theorem 7]{beltita05integrability}), we obtain the following:
\begin{corollary}\label{integrability}
Let $G/K$ be a homogeneous space equipped with a homogeneous almost complex structure $\J$ given by $\V\in\lcgk$. Then $\J$ is integrable (i.e. $G/K$ admits complex charts compatible with $\J$) if and only if it is Nijenhuis, i.e. if and only if
\[
Z_{+}=\{v\in \mathfrak g^{\mathbb C}: (\V^{\mathbb{C}} -  i)v\in 
 \mathfrak k^{\mathbb C}\},
\]
is a complex Lie subalgebra of $\mathfrak{g}^{\mathbb{C}}$.
\end{corollary}
\medskip

\begin{remark} By conjugation, 
$Z_+$ is a complex Lie subalgebra of $\mathfrak g^{\mathbb{C}}$ if and only if $Z_-$ is.
\end{remark}

\begin{remark}\label{remsupl}
If $K$ is split in $G$, i.e. if $\mathfrak k$ has a closed complement $\mathfrak m$ in $\mathfrak g$, $\mathfrak g=\mathfrak k\oplus \mathfrak m$,  one can consider $\V \in \mathcal{I}_{\mathfrak m}(G, K)$ (see Remark~\ref{more_general_complex}). In this case the spaces $Z_{\pm}$ defined in Definition~\ref{invs}  are given by 
$$
Z_{\pm} = \mathfrak k^{\mathbb{C}} \oplus \operatorname{Eig_{\pm i}}(\V^{\mathbb{C}}_{|\mathfrak m^{\mathbb{C}}}),
$$
where $\operatorname{Eig_{\pm i}}(\V^{\mathbb{C}}_{|\mathfrak m^{\mathbb{C}}})$ is the eigenspace with eigenvalue $\pm i$ of the complex linear extension $\V^{\mathbb{C}}$ restricted to $\mathfrak m^{\mathbb{C}}$.
Note that in this case, we have
\begin{itemize}
    \item $\mathfrak g^{\mathbb C} = Z_+ + Z_-$
    \item $Z_+\cap Z_- = \mathfrak k^{\mathbb C}$
    \item $\Ad_{\mathfrak k} Z_\pm \subset Z_\pm$.
\end{itemize}
By Theorem~15 in \cite{beltita05integrability} (see also \cite{frolicher55}), in this complemented case, any homogeneous complex structure on $M = G/K$ comes from this kind of decomposition of $\mathfrak g^{\mathbb C}$. 
In particular, in the complemented case, Corollary~\ref{integrability} reduces to Theorem~13 in \cite{beltita05integrability}.
\end{remark}

\medskip

\section{Examples}

\begin{example} We now return to our example of the sphere $\mathbb{S}^2\simeq SO(3)/SO(2)$ presented in Subsection \ref{thesphere}. Let $k_0,e_1,e_2\in \mathfrak{so}(3)$ be as in  Subsection \ref{thesphere}, and note that from the definition $\I k_0 = 0$, $\I e_1 = -e_2$, $\I e_2=e_1$ it follows that $\Ran(\I^2+1)=\{0\}\subset\mathfrak k$ therefore $\I$ induces an almost complex structure $\N$ on $\mathbb{S}^2\simeq SO(3)/SO(2)$. By Remark \ref{remsupl}, the vanishing of its torsion is equivalent to
\[
Z_+=\{E_1+iE_2\}
\]
being a Lie subalgebra of $\mathfrak{so}(3)$, which is trivial because it is (complex) one-dimensional. Therefore this particular complex structure is integrable (that the sphere $\mathbb{S}^2$ is a complex manifold is of course well-known).  It is known that the real sphere $\mathbb{S}^n$ admits an almost complex structure if and only if $n=2$ or $n=6$ (see \cite{borel} or the survey \cite{konstantis18} for further details). 
The known almost complex structure on $\mathbb{S}^6$ is also homogeneous and can be constructed by considering $\mathbb{S}^6$ inside the subspace of purely imaginary octonions; however this almost complex structure is not integrable (see \cite{konstantis18}). For the infinite dimensional sphere $\mathcal S$ (the unit sphere of a real Hilbert space $\H$), it is known that $\mathcal S$ is real analytic isomorphic to $\H$, see \cite{dobrowolski95}, therefore $\mathcal S$ admits an almost complex structure, being a complex manifold (an infinite dimensional real Hilbert space $\H$ is also a complex Hilbert space, halving the basis).

Going back to $\mathbb{S}^2$, consider the following linear operator
\[\I k_0 = \alpha k_0,\qquad \I e_1 = \beta e_2, \qquad \I e_2=\gamma e_1\]
for  $\alpha,\beta,\gamma\in\R$. It is admissible, i.e. $\I\in\lgk$, if and only if $\gamma=-\beta$ and then it gives rise to a homogeneous vector bundle map $\N$ on $\mathbb{S}^2$. The condition \eqref{localtor} also holds in this case, which means that $\N$ is a Nijenhuis operator on $\mathbb{S}^2$. Additionally $\I\in\lcgk$ only if $\beta=\pm 1$, in which case we obtain again an integrable homogeneous complex structure, which is up to sign equal to the previously described one.
\end{example}

Let us now consider an example of an infinite-dimensional homogeneous space, where we will introduce a vector bundle map and prove it is not a Nijenhuis operator by means of Theorem \ref{torsionJ}.

\begin{example}\label{ex:compact}
Let $\H$ be an infinite dimensional Hilbert space, denote with $\mathcal B(\H)$ the bounded linear operators acting in $\H$, with $\mathcal K(\H)$ the ideal of compact operators, and consider the group of invertible operators $G=GL(\H)\subset \mathcal B(\H)$. For the Lie subgroup $K$ consider the group of invertible operators which differ from the identity by a compact operator $K=GL(\H)\cap(\Id+\mathcal K(\H))$. Note that since $\mathcal K(\H)$ is a closed subspace of $\mathcal B(\H)$, then $K$ is an immersed subgroup of $G$ (moreover, it is embedded since the topology of $K$ is the norm topology). But $K$ is not split in $G$, since the compact operators are not complemented in the bounded operators. Now, since compact operators are a closed ideal in the algebra of bounded operators, the group $K$ is a normal subgroup of $G$, therefore the quotient has a structure of Banach--Lie group, which makes of the quotient map $\pi:G\to G/K$ a smooth submersion (see \cite[Theorem II.2]{glockner-neeb03}).

Consider the linear functional  $\ell$ on $S=\mathbb C \Id +\mathcal K(\H)$ defined as $\ell(\mathcal K(\H))=0$ and $\ell(\Id)=1$. 
Since 
\[
\|t \Id + k \|=|t| \|\Id +k'\| \ge |t| d =|\ell(t\Id + k)|\,  d
\]
where $d=dist(\Id,\mathcal K(\H))=\inf_{k\in \mathcal K(\H)} \|k+\Id\|>0$, 
it follows that $\ell$ is bounded in $S=\mathbb C \Id +\mathcal K(\H)$. By means of the Hahn--Banach theorem one extends it to a bounded functional on the whole $\mathcal B(\H)$, also denoted by $\ell$. Now consider $\I\in\mathcal B(\mathcal B(\H))$ given by
\[ \I(X) = \ell(X) \cdot \Id. \]
Let us verify that indeed $\I\in\lgk$.
By definition it vanishes on $\mathcal K(\H)$, so it preserves it in a trivial manner. The other condition is equivalent to  
\begin{equation}\label{commute} 
\Ad_{\Id+k}\I(X) - \I(\Ad_{\Id+k}X)\in \mathcal K(\H)
\end{equation}
for all $k\in \mathcal K(\H)$ such that $\Id+k\in GL(\H)$, $X\in\mathcal B(\H)$. Since  $\I(X)$ lies in the center of $\mathcal B(\H)$, condition~\eqref{commute}   can be written as 
\[\I(X) - \I((\Id+k)X(\Id+\tilde k))\in\mathcal K(\H),\]
where $\Id+\tilde k = (\Id+k)^{-1}$ with $\tilde{k}$ compact. Since $kX$, $X\tilde k$ and $k X\tilde k$ are all compact and hence in the kernel of $\I$, the identity holds.

Thus $\I$ defined in this way gives rise to a homogeneous vector bundle map $\N$ on $G/K$. However by Theorem~\ref{torsionJ} it is never a Nijenhuis operator. Namely in the condition \eqref{localtor} the first three terms vanish identically since the image of $\I$ lies inside the center of the Lie algebra $\mathcal B(\H)$. 
Thus the condition for $\N$ to be Nijenhuis is
\[ \I^2([v,w]) \in  \mathcal K(\H) \qquad \textnormal{for all } v,w\in\mathcal B(\H).\]
Note that  by definition $\I$ is idempotent and  never takes value in $\mathcal K(\H)\setminus\{0\}$. Thus for $\N$ to be Nijenhuis the following identity should be satisfied
\begin{equation}\label{tor-hilbert}\I([v,w]) = 0 \qquad \textnormal{for all } v,w\in\mathcal B(\H).\end{equation}
It was demonstrated in \cite{halmos-commutators} that every operator in $\mathcal B(\H)$ is the sum of four commutators, thus the linear span of all commutators is equal to the whole $\mathcal B(\H)$. Thus the condition \eqref{tor-hilbert} never holds, as it would imply $\I=0$.

Let us also mention that vector bundle maps $\I$ of the discussed form never give rise to an almost complex structure since they are idempotent $\I^2=\I$.
\end{example}

\begin{example}
Consider $G$ and $K$ as in previous example. Let's look for another simple case of the operator $\I$: namely right and left multiplication by bounded operators:
\[ \I(X)= AXB, \]
for $A,B,X\in \mathcal B(\H)$. In this case the condition $N\mathcal K(\H)\subset \mathcal K(\H)$ is automatically satisfied since $\mathcal K(\H)$ is a two-sided ideal in $\mathcal B(H)$.

The other condition for $\I$ to be admissible (see Remark~\ref{connected}) is
\[
[k,AXB] - A[k,X]B \in \mathcal K(\H),
\]
for $k \in \mathcal K(\H)$.
It is also automatically satisfied for the same reason. Thus the operator $\I$ descends to the operator $\N$ on the homogeneous space. If we choose $A$ and $B$ such that $A^2=B^2=-1$ we get an almost complex structure on $G/K$.

Let us verify using Theorem~\ref{torsionJ} if it is a Nijenhuis operator. The condition \eqref{localtor} doesn't hold in general, however in a simpler case when either $A=\Id$ or $B=\Id$ it is always satisfied. Thus left (or right) multiplication by a bounded operator from $\mathcal B(\H)$ gives always rise to a Nijenhuis operator on $G/K$. If we choose this operator in a such way that it's square is $-1$, we obtain an integrable complex structure on $G/K$.
\end{example}

\begin{example}
Let us mention here another well-known example. Consider a separable infinite-dimensional complex Hilbert space $\H$ endowed with the orthogonal decomposition 
\[ \H = \H_-\oplus \H_+ \]
onto two infinite-dimensional closed subspaces.
Denote by $P_\pm$ an orthogonal projection onto $\H_\pm$ and by $d=i(P_+-P_-)$.

Consider the Banach--Lie group $G=U_{\textnormal{res}}$ defined as follows (see e.g. \cite{segal}):
\[
U_{\textnormal{res}} = \{U\in \mathcal B(\H)\;|\; U^*U=UU^*=\Id, [U,d]\in L^2(\H)\},
\]
where $L^2(\H)$ is the ideal of Hilbert--Schmidt operators. Its Banach--Lie algebra is 
\[
\mathfrak g = \mathfrak u_{\textnormal{res}} = \{u\in \mathcal B(\H)\;|\; u^*=-u, [u,d]\in L^2(\H)\}.
\]
One verifies readily that $d\in\mathfrak u_{\textnormal{res}}$.

The group $U_{\textnormal{res}}$ acts on the Hilbert space $\H$ in the natural way and in consequence it also acts on the Grassmannian of $\H$, i.e. the set of all closed subspaces of $\H$. The action on the Grassmannian is not transitive. The orbit of the closed infinite-dimensional subspace $\H_+$ is known as the \emph{restricted Grassmannian} $Gr_{\textnormal{res}}$. The stabilizer of $\H_+$ is a product of two unitary groups $K=\mathcal U(\H_+)\times\mathcal U(\H_-)$.
The restricted Grassmannian $Gr_{\textnormal{res}}$ possess a manifold structure and $\pi$ is a submersion. It is thus a homogeneous space $G/K = U_{\textnormal{res}}/(\mathcal U(\H_+)\times\mathcal U(\H_-))$.

We can define  Nijenhuis operators on $Gr_{\textnormal{res}}$ by considering  bounded operators on the Banach--Lie algebra $\mathfrak u_{\textnormal{res}}$ of the form $\I = \ad_{\tilde d}$, where $\tilde d$ belongs to the center of $\mathfrak k = \mathfrak u(\H_+)\times\mathfrak u(\H_-)$. It is easy to check that $\ad_{\tilde d}\in\lgk$ and thus it descends to a vector bundle map $\N$ on $TGr_{\textnormal{res}}$.
Let us employ Theorem \ref{torsionJ} to verify that the torsion of $\N$ vanishes.  By Corollary \ref{cor:ad_nij}, $\N$ is a Nijenhuis operator if and only if
\[ \big[ [\tilde d, v], [\tilde d, w] \big] \in \mathfrak u(\H_+)\times\mathfrak u(\H_-) \]
for all $v,w\in \mathfrak u_{\textnormal{res}}$.
This can be checked by direct computation, and is equivalent to the fact that the restricted Grassmannian is a (locally) symmetric space.

The restricted Grassmannian $Gr_{\textnormal{res}}$ is a K\"ahler manifold, which means among others that it possesses a complex structure. It is induced by $\I =\ad_{\tilde d}$  with  $\tilde d = \frac12 d = \frac12 i (P_+ - P_-)$. Direct computation shows that indeed $\Ran((\ad_{\tilde d})^2 + 1)\subset\mathfrak u(\H_+)\times\mathfrak u(\H_-)$ thus we  obtain an almost complex structure. Previous considerations prove that it is indeed integrable. Let us note that the restricted Grassmannian is also a symplectic leaf in a certain Banach Lie--Poisson space (central extension of the predual space of $\mathfrak u_{\textnormal{res}}$, see \cite{Ratiu-grass})  and it leads to a hierarchy of integrable systems on it, see \cite{GO-grass,GT-momentum}.
\end{example}
\begin{remark}
In finite dimensions there is a well known method of obtaining an almost complex structure on coadjoint orbits of Lie groups, see e.g. \cite[Section~1.2, Theorem 2]{marsden-chernoff},\cite[Part V-Section 12.2]{cannas2001} and \cite{weinstein1977lectures}. It goes by considering a polar decomposition of the $\ad_d$ operator. In the paper \cite{GLT-kahler} the generalization of this approach will be applied to the study of unitary orbits of trace-class operators, in the spirit of Kirillov's orbit method \cite{kirillov-lec}. The results of the present paper will be used to address the question of integrability of these structures.   
\end{remark}


\bigskip

\section*{Acknowledgements}
This research was partially supported by ANPCyT, CONICET, UBACyT  20020220400256BA (Argentina) and 2020 National Science Centre, Poland/Fonds zur Förderung der wissenschaftlichen Forschung, Austria grant ``Banach Poisson--Lie groups and integrable systems'' 2020/01/Y/ST1/00123, I-5015N.


\end{document}